\documentclass{article}
\usepackage[utf8]{inputenc}

\usepackage{amsmath}
\usepackage{amssymb}
\usepackage{amsfonts}
\usepackage{amsthm}
\usepackage{mathtools}

\usepackage{graphicx}
\usepackage{tikz-cd}
\usepackage{tikz-3dplot}

\usepackage[normalem]{ulem}
\usepackage[hyphens]{url}
\usepackage{hyperref}
\usepackage[capitalise]{cleveref}

\usepackage{authblk}
\theoremstyle{plain}
\newtheorem{theorem}{Theorem}
\newtheorem{lemma}{Lemma}
\newtheorem{proposition}{Proposition}
\newtheorem{corollary}{Corollary}

\theoremstyle{definition}
\newtheorem{definition}{Definition}
\newtheorem{example}{Example}

\theoremstyle{remark}
\newtheorem{remark}{Remark}

\newcommand{\spec}{\operatorname{D}}
\newcommand{\wi}{\cong^w} 
\newcommand{\fms}{\mathbf{FMS}} 
\newcommand{\nat}{\Longrightarrow}
\renewcommand{\implies}{\Rightarrow}
\renewcommand{\impliedby}{\Leftarrow}

\usepackage{mathrsfs}
\newcommand{\sif}{\mathscr{I}}

\newcommand{\vr}{\mathrm{VR}} 

\title{On the notion of weak isometry for finite metric spaces}

\author[1]{Alessandro De Gregorio\thanks{alessandro.degregorio@polito.it}}
\author[1]{Ulderico Fugacci\thanks{ulderico.fugacci@polito.it}}
\author[2]{Facundo Memoli\thanks{memoli@math.osu.edu}}
\author[1]{Francesco Vaccarino\thanks{francesco.vaccarino@polito.it}}

\affil[1]{Department of Mathematical Sciences, Politecnico di Torino}
\affil[2]{Department of Mathematics and Department of Computer Science and Engineering, The Ohio State University}

\date{May 2020}

\begin{document}

\maketitle

\begin{abstract}
Finite metric spaces are the object of study in many data analysis problems.
We examine the concept of \emph{weak isometry} between finite metric spaces, in order to analyse properties of the spaces that are invariant under strictly increasing rescaling of the distance functions. In this paper, we analyse some of the possible complete and incomplete invariants for weak isometry and we introduce a dissimilarity measure that asses how far two spaces are from being weakly isometric. Furthermore, we compare these ideas with the theory of persistent homology, to study how the two are related.
\end{abstract}

\section{Introduction}
\label{sec:intro}
Finite metric spaces arise in many applicative problems, whenever we have a set of objects on which it is defined a measure of dissimilarity. One of the main purposes of Topological Data Analysis (TDA) is to gather topological and geometric information from these datasets \cite{carlsson09}, in order to be able to perform comparisons between the phenomena that generated the measurements. The main input, in the TDA pipeline, is a nested sequence of simplicial complexes called \emph{filtration}. Such a sequence is usually obtained from a weighted undirected network, or, in case the weights of the network satisfy the triangular inequality, from a finite metric space. In many problems, it is interesting to obtain information which is invariant under ``rescalings'' of the metric space. If, for example, the distances are physical quantities, we may want that the results that we get are independent from the system used for measurements. Sometimes, our observations undergo transformations that are not linear, but that preserve the order of distance between points. For example, in \cite{giusti15}, Giusti er al. showed that such an approach is useful to study data from neural activity and connectivity.
In the recent theory of monoid equivariant operators \cite{chacholski20}, an interesting family of operators is that of the so called \emph{change of units}, that are nothing but functions that transform a dataset via a \emph{rescaling} of the observations.
In this work, we want to give a definition of weak isometry that let us identify two finite metric spaces if one of the distance functions can be obtained from the other by mean of a composition with a strictly increasing real valued function. We compare our approach to the work of Ganyushkin and Tsvirkunov \cite{ganyushkin94}, where they perform a classification of finite metric spaces and introduce a notion of isomorphism between them. 
We will show that the problem of determining if two metric spaces are weakly isometric or not can be reduced to the classical problem of isometry. Then, we will see how curvature sets introduced by Gromov \cite{gromov07} and Viertoris-Rips filtration can serve as complete and incomplete invariants for weak isometry.

\section{Weakly isometric finite metric spaces}
\label{sec:weak-isometry}
Henceforth we will consider all the metric spaces taken into account to be finite. We will write FMS for finite metric space. Whenever we write $\mathbb{R}^+$ we are considering the set $\{x\in \mathbb{R}\ |\ x\geq 0\}$.

\begin{definition}[weak isometry]
\label{def:weak-isometry}
Let us consider two metric spaces $(X,d_X)$ and $(Y,d_Y)$. We say that they are weakly isometric if there exist a bijection $\varphi:X \longrightarrow Y$ and a strictly increasing function $\psi: \mathbb{R}^+\longrightarrow\mathbb{R}^+$ such that, for all $x_1, x_2 \in X$,
\begin{equation}
    \label{eq:def:weak-isometry}
    \psi(d_X(x_1,x_2)) = d_Y(\varphi(x_1),\varphi(x_2)).
\end{equation}

If $(X,d_X)$ and $(Y,d_Y)$ are weakly isometric, we will write $$(X,d_X) \wi (Y,d_Y).$$
\end{definition}
This concept is also defined in \cite{dovgoshey13}, called by the authors \emph{weak similarity}, where it is applied to semi-metric spaces of possibly infinite cardinality.

\begin{theorem}
\label{thm:wy-equivalence-relation}
The relation of weak isometry is an equivalence relation.
\end{theorem}
\begin{proof}
We check the three properties of equivalence relations.
\begin{itemize}
    \item Reflexivity: using the identity, $X \wi X$ for all finite metric spaces $X$.
    \item Symmetry: if $X \wi Y$ we have a bijection $\varphi:X \longrightarrow Y$ and a strictly increasing function $\psi: \mathbb{R}^+\longrightarrow\mathbb{R}^+$ with $\psi(d_X(x_1,x_2)) = d_Y(\varphi(x_1),\varphi(x_2))$ for all $x_1,x_2 \in X$. We have that $\psi$ is an invertible function since it is strictly monotone. For each $y_1,y_2\in Y$ consider $x_1,x_2\in X$ with $y_i = \varphi(x_i)$, where $i=1,2$. Then,
    \begin{equation*}
        d_X(\varphi^{-1}(y_1),\varphi^{-1}(y_2)) = \psi^{-1}(d_Y(y_1,y_2)).
    \end{equation*}
    and so, by definition, $Y \wi X$.
    \item Transitivity: if $X \wi Y$ and $Y \wi Z$ consider the functions $\varphi_1, \varphi_2, \psi_1, \psi_2$ such that
    \begin{align*}
        \psi_1(d_X(x_1,x_2)) &= d_Y(\varphi_1(x_1),\varphi_1(x_2)) \quad \forall x_1,x_2\in X,\\
        \psi_2(d_Y(y_1,y_2)) &= d_Z(\varphi_2(y_1),\varphi_2(y_2)) \quad \forall y_1,y_2\in Y.
    \end{align*}
    Then,
    \begin{equation*}
        \psi_1(d_X(x_1,x_2)) = d_Y(\varphi_1(x_1),\varphi_1(x_2)) = \psi_2^{-1}(d_Z(\varphi_2(\varphi_1(x_1)),\varphi_2(\varphi_1(x_2))))
    \end{equation*}
    hence, considering the functions $\varphi_2 \circ \varphi_1$ and $\psi_2 \circ \psi_1$, we have $X \wi Z$.
\end{itemize}
\end{proof}

Ganyushkin and Tsvirkunov \cite{ganyushkin94}  introduced a notion of isomorphism for finite metric spaces which we will compare to weak isometry below. 

\begin{definition}[isomorphism - \cite{ganyushkin94}]
\label{def:isomorphism}
We say that two metric spaces $(X,d_X)$ and $(Y,d_Y)$ are isomorphic if there is a bijection $\varphi:X\longrightarrow Y$ such that for all $x_1, x_2, x'_1, x'_2 \in X$ we have

\begin{align}
    \label{eq:def:isomorphism-1}
    d_X(x_1,x_2) = d_X(x'_1,x'_2) \implies d_Y(\varphi(x_1),\varphi(x_2)) = d_Y(\varphi(x'_1),\varphi(x'_2)) \\
    \label{eq:def:isomorphism-2}
    d_X(x_1,x_2) < d_X(x'_1,x'_2) \implies d_Y(\varphi(x_1),\varphi(x_2)) < d_Y(\varphi(x'_1),\varphi(x'_2)).
\end{align}
If $(X,d_X)$ and $(Y,d_Y)$ are isomorphic we will write $(X,d_X) \simeq (Y,d_Y)$.
\end{definition}
Now, we see how the two concepts are related.
\begin{theorem}
Two finite metric spaces are weakly isometric if and only if they are isomorphic.
\end{theorem}
\begin{proof}
Assume that $X \wi Y$. Then, we have two functions $\varphi$ and $\psi$ such that, for all $x_1,x_2,x'_1, x'_2 \in X$,

\begin{align}
    \psi(d_X(x_1,x_2)) &= d_Y(\varphi(x_1),\varphi(x_2)), \\
    \psi(d_X(x'_1,x'_2)) &= d_Y(\varphi(x'_1),\varphi(x'_2)).
\end{align}
Then, if $d_X(x_1,x_2) = d_X(x'_1,x'_2)$, we have
\begin{equation}
    d_Y(\varphi(x_1),\varphi(x_2))  =  \psi(d_X(x_1,x_2)) = \psi(d_X(x'_1,x'_2)) = d_Y(\varphi(x'_1),\varphi(x'_2)).
\end{equation}
If $d_X(x_1,x_2) < d_X(x'_1,x'_2)$, since $\psi$ is a strictly increasing function,
\begin{equation}
    d_Y(\varphi(x_1),\varphi(x_2))  =  \psi(d_X(x_1,x_2)) < \psi(d_X(x'_1,x'_2)) = d_Y(\varphi(x'_1),\varphi(x'_2)).
\end{equation}
Therefore, $X \wi Y \implies X \simeq Y$.\\ 
On the other hand, assume $X \simeq Y$ and consider the bijection $\varphi$ given by  \cref{def:isomorphism}. It is possible to order all the pairs $(x_i,x_j) \in X\times X$ so that 
\begin{equation}
    d_X(x_{i_1}, x_{j_1}) \leq  d_X(x_{i_2}, x_{j_2})\leq \dots \leq  d_X(x_{i_{n^2}}, x_{j_{n^2}}),
\end{equation}
where $n$ is the number of points of $X$ and $Y$. Because of implications \eqref{eq:def:isomorphism-1} and \eqref{eq:def:isomorphism-2}, we have that
\begin{equation}
    d_Y(\varphi(x_{i_1}), \varphi(x_{j_1})) \leq  d_Y(\varphi(x_{i_2}), \varphi(x_{j_2}))\leq \dots \leq  d_Y(\varphi(x_{i_k}), \varphi(x_{j_k})),
\end{equation}
hence, we can define an increasing function $\psi:\mathbb{R}^+ \longrightarrow \mathbb{R}^+$ with
\begin{equation}
    \psi( d_X(x_{i_r},x_{j_r})) = d_Y(\varphi(x_{i_r}),\varphi(x_{j_r}))\quad \forall (x_{i_r},x_{j_r}) \in X\times X
\end{equation}
and then $X \simeq Y \implies X \wi Y$.
\end{proof}

We have seen that the two concepts are the same, but weak isometry explicitly shows the rescaling that has to be performed to obtain one space from the other. Now, we want to associate to each equivalence class of weak isometry a good representative.

\begin{definition}[distance set]
\label{def:spectrum}
Given a metric space $(X,d)$, we define its distance set $\spec(X,d)$ as the set of all pairwise distances between different points of $X$.
\begin{equation}
    \label{eq:def:spectrum}
    \spec(X,d) := \left\{ d(x_1,x_2)\ \middle|\ x_1, x_2 \in X,\ x_1 \neq x_2  \right\}.
\end{equation}

When there is no ambiguity for the metric $d$ defined on the space $X$, we will simply write $\spec(X)$.
\end{definition}

\begin{lemma}
\label{lem:spectrum-of-isomorphic-fms}
Two weakly isometric finite metric spaces have distance sets of the same cardinality.
\end{lemma}
\begin{proof}
If $X \wi Y$, there is a strictly increasing function $\psi$ such that 
\begin{equation}
    \psi(\spec(X)) := \left\{ \psi(l)\ \middle|\ l \in \spec(X) \right\} = \spec(Y).
\end{equation}
So, since $\psi\vert_{\spec(X)}$ is injective by definition and surjective on $\spec(Y)$, then $\left|\spec(X)\right| = \left|\spec(Y)\right|$.
\end{proof}
\begin{remark}
The reciprocal statement does not hold. Two spaces can have the same distance set but not be weakly isometric. For example, Boutin and Kemper in \cite{boutin04} study the problem of recontruction of a metric space given the distribution of distances between points.
\end{remark}

\begin{example}
The two metric spaces in  \cref{fig:non-wi-fms}, $X = \{a,b,c\}$ with $d_X(a,c) = d_X(b,c)= 6$, $d_X(a,b)=5$ and $Y = \{d,e,f\}$ with $d_Y(d,f) = d_X(e,f)= 5$, $d_Y(d,e)=6$ have the same distance set, but they are not weakly isometric.
\tikzset{every picture/.style={line width=0.75pt}} 
\begin{figure}[ht!]
\centering
    \begin{tikzpicture}[x=0.75pt,y=0.75pt,yscale=-1,xscale=0.75]
    \draw   (168.5,29.99) -- (247.84,242.96) -- (86.53,242.96) -- cycle ;
    \draw   (474.25,151) -- (605.5,240) -- (343,240) -- cycle ;
    \draw (113,134) node  [align=left] {6};
    \draw (165,261) node  [align=left] {5};
    \draw (231,136) node  [align=left] {6};
    \draw (75,252) node  [align=left] {a};
    \draw (264,250) node  [align=left] {b};
    \draw (169,17) node  [align=left] {c};
    \draw (338,252) node  [align=left] {d};
    \draw (620,251) node  [align=left] {e};
    \draw (477,138) node  [align=left] {f};
    \draw (474,252) node  [align=left] {6};
    \draw (407,180) node  [align=left] {5};
    \draw (543,183) node  [align=left] {5};
    \end{tikzpicture}
\caption{FMS that are not weakly isometric.}
\label{fig:non-wi-fms}
\end{figure}
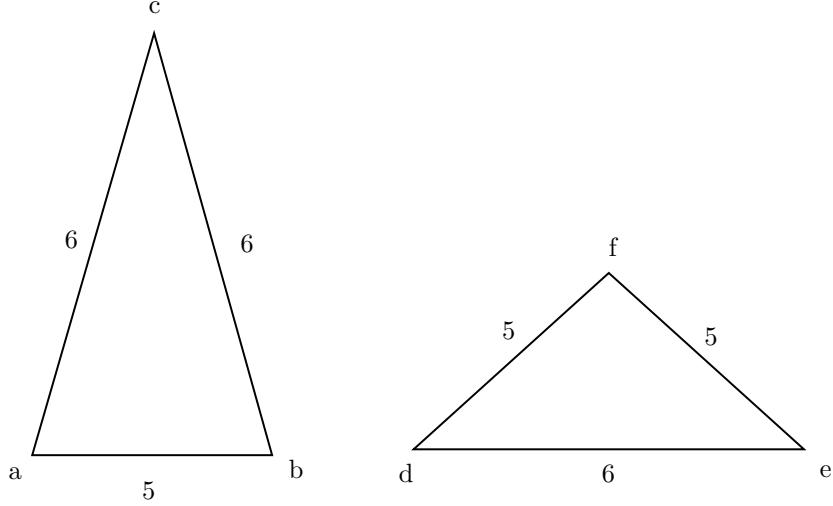
\end{example}

\begin{definition}[natural-valued metric space]
\label{def:natural-valued-ms}
We say that the metric $d$ defined on the space $X$ is natural-valued if $\spec(X) \subset \mathbb{N}$.
\end{definition}

\begin{definition}[dense distance set]
\label{def:dense-spectrum}
Given a finite metric space $(X,d)$ with natural-valued matric $d$ we say that its distance set is dense if it is a list of consecutive natural numbers, 
\begin{equation}
\label{eq:def:dense1}
    \spec(X) = \left\{a+1, a+2, \dots, a+|\spec(X)| \right\}.
\end{equation}
\end{definition}

\begin{definition}[canonical]
\label{def:canonical}
A finite metric space of cardinality $n$ is called $k$-canonical if it has a natural valued and dense distance set of the form 
\begin{equation}
    D_{n,k}:= \left\{ 2\binom{n}{2}, 2\binom{n}{2}-1,\dots, 2\binom{n}{2}-k+1\right\}.
\end{equation}
\end{definition}
Canonical metric spaces are interesting because for them the notion of weak isometry is equivalent to that of isometry, as we show in the next Lemma.
\begin{lemma}
\label{lem:isometry-of-canonicals}
Let $\mathcal{C}_1$ and $\mathcal{C}_2$ be canonical finite metric spaces. Then, $\mathcal{C}_1$ is weakly isometric to $\mathcal{C}_2$ if and only if $\mathcal{C}_1$ is isometric to $\mathcal{C}_2$.
\end{lemma}
\begin{proof}
If $\mathcal{C}_1$ is isometric to $\mathcal{C}_2$, then $\mathcal{C}_1$ is weakly isometric to $\mathcal{C}_2$. On the other hand assume $\mathcal{C}_1 \wi \mathcal{C}_2$. By  \cref{lem:spectrum-of-isomorphic-fms}, we know that the two spaces have distance sets of the same cardinality. On the other hand, the distance set of a canonical space is defined by its cardinality, so 
\begin{equation}
\begin{split}
\spec(\mathcal{C}_1) &= \left\{ 2\binom{n}{2}, 2\binom{n}{2} - 1,\dots, 2\binom{n}{2} -|\spec(\mathcal{C}_1)|+1 \right\} = \\
 &=\left\{ 2\binom{n}{2}, 2\binom{n}{2} - 1,\dots, 2\binom{n}{2} -|\spec(\mathcal{C}_2)|+1 \right\} = \spec(\mathcal{C}_2).
\end{split}
\end{equation}
By hypothesis, we have a bijection $\varphi:\mathcal{C}_1 \longrightarrow \mathcal{C}_2$ and a strictly increasing function $\psi:\mathbb{R}^+ \longrightarrow \mathbb{R}^+$ such that $\psi(\spec(\mathcal{C}_1)) = \spec(\mathcal{C}_2)$. Since the two distance sets are equal, such a function can only be the identity and, therefore, 
\begin{equation}
    \psi(d_{\mathcal{C}_1}(x_1,x_2))= d_{\mathcal{C}_1}(x_1,x_2) =  d_{\mathcal{C}_2}(\varphi(x_1),\varphi(x_2))
\end{equation}
and the two spaces are isometric.
\end{proof}
In Proposition 2 of \cite{ganyushkin94}, the authors proved the following useful result.
\begin{theorem}
\label{thm:ganyushkin-canonical}
Each finite metric space is isomorphic to a canonical metric space.
\end{theorem}

\begin{remark}
We will say that a finite metric space $\mathcal{C}$ is canonical for the FMS $X$ if they are weakly isometric and $\mathcal{C}$ is canonical.
\end{remark}
We can now define a map called \emph{canonicalization} which assigns to each finite metric space $X$ a canonical metric space $\mathcal{C}_X$.
\begin{definition}[canonicalization]
\label{def:canonicalization}
Let $X$ be a finite metric space of cardinality $n$ and distance set $\spec(X) = \{a_1, a_2, \dots, a_k\ | a_i<a_i \textrm{ if }i<j\}$. Define $\psi: \mathbb{R}^+ \longrightarrow \mathbb{N}$ as
\begin{equation}
    \psi(a_i) = 2\binom{n}{2} - k + i.
\end{equation}
The canonicalization of $X$ is the space $(\mathcal{C}_X, d_{\mathcal{C}_X})$ with:
\begin{itemize}
    \item $\mathcal{C}_X = X$ 
    \item $
        d_{\mathcal{C}_X}(x_1,x_2) = \begin{cases}
        \psi(d_X(x_1,x_2)) &\textrm{ if } x_1\neq x_2 \\
        0 &\textrm{ if } x_1= x_2. 
        \end{cases}
    $
\end{itemize}
\end{definition}
Canonical metric spaces are important because, as we will see with the following corollary, they allow us to associate to each weak isometry equivalence class a unique representative. 
\begin{corollary}[uniqueness of canonical representations]
\label{cor:unicity-of-canonical}
Let $X$ be a finite metric space with $\mathcal{C}$ and $\mathcal{C}'$ canonical for $X$. Then, $\mathcal{C}$ and $\mathcal{C}'$ are isometric, that is, there is a unique metric space canonical for $X$ up to isometry.
\end{corollary}
\begin{proof}
By the transitivity of weak isometry, we have that $\mathcal{C} \simeq X \simeq \mathcal{C}'$ and, by  \cref{lem:isometry-of-canonicals}, they are isometric.
\end{proof}

For a given FMS $X$,  \cref{cor:unicity-of-canonical} allows us to identify a canonical representative of a  $X$ in the form of its canonicalization, as introduced in  \cref{def:canonicalization}.

In this way, we can now reformulate the problem of weak isometry to that of classical isometry between canonical spaces.
\begin{theorem}
\label{thm:wi-to-isometry}
Two finite metric spaces are weakly isometric if and only if their canonicalizations are isometric.

\end{theorem}
\begin{proof}
If $\mathcal{C}$ is canonical for $X$ and $Y$, then $X \wi \mathcal{C} \wi Y$ and, hence, $X$ and $Y$ are weakly isometric. On the other hand, assume that $X$ is weakly isometric to $Y$. By the transitivity of weak isometry, we have
\begin{equation*}
    \mathcal{C}_X \wi X \wi Y \wi \mathcal{C}_Y,
\end{equation*}
therefore, $\mathcal{C}_X$ and $\mathcal{C}_Y$ are weakly isometric. By  \cref{lem:isometry-of-canonicals}, we have that they have to be isometric.
\end{proof}

\section{A dissimilarity measure for weak isometry}
\label{sec:dissimilarity}
We would like to define a pseudo-distance between finite metric spaces that measures how far are two metric spaces from being weakly isometric. We are looking for a pseudo-distance $d: \mathrm{FMS} \times \mathrm{FMS} \longrightarrow \mathbb{R}$ such that:
\begin{enumerate}
    \item $d((X,d_X),(Y,d_Y))=0$ if and only if $(X,d_X)$ and $(Y,d_Y)$ are weakly isometric,
    \item $d$ is ``continuous'', that means that $d$ is not discrete and takes into account the actual values assumed by the metrics $d_X$ and $d_Y$, and does not discriminate them only because of their ``combinatorial properties''.
\end{enumerate}
\begin{proposition}
There is no pseudo-distance that satisfies the two conditions above.
\end{proposition}
\begin{proof}
Let us assume that the pseudo-distance $d$ is such that it is zero if and only if two spaces are weakly isometric. 

Since $d((X,d_X),(Y,d_Y))=0$ and $d$ satisfies the triangle inequality, for any three spaces $(X,d_X)$, $(Y,d_Y)$ and $(Z,d_Z)$ with $(X,d_X)$ weakly isometric to $(Y,d_Y)$ we have that
\begin{equation}
    d((X,d_X),(Z,d_Z)) = d((Y,d_Y),(Z,d_Z)).
\end{equation}
Therefore, the function $d$ is constant on the equivalence classes given by weak isometry, hence it is discrete.
\end{proof}
Since it is not possible to find such a pseudo-distance, we will weaken our demands by dropping the assumption that the function $d$ satisfies the triangle inequality. 

\begin{definition}[dissimilarity measure]
\label{def:diss}
A dissimilarity measure $d$ over a set $S$ is a function $d:S\times S\longrightarrow \mathbb{R}$ such that: 
\begin{enumerate}
    \item there exists a number $d_0\in\mathbb{R}$, with $-\infty<d_0\leq d(s_1,s_2)<+\infty$, for all $s_1,s_2\in S$, and $d(s,s)=d_0$, for all $s\in S$,
    \item $d(s_1,s_2) = d(s_2,s_1)$, for all $s_1,s_2\in S$.
\end{enumerate}
\end{definition}

Since we will make use of the Gromov-Hausdorff distance, we recall its definition.
\begin{definition}[Gromov-Hausdorff distance]
\label{def:gh-distance}
Given two metric spaces $(X,d_X)$ and $(Y,d_Y)$, a \emph{correspondence} between them is a set $C\subseteq X\times Y$ such that $\pi_X(C) = X$ and $\pi_Y(C) = Y$, where $\pi_X$ and $\pi_Y$ are the canonical projections of the product space. We denote with $\mathfrak{C}(X,Y)$ the set of all correspondences between $X$ and $Y$. We define the \emph{distortion} of a correspondence $C$, with respect to the metrics $d_X$ and $d_Y$ as
\begin{equation}
    \label{eq:distortion}
    \operatorname{dis}(C,d_X,d_Y) = \sup_{(x,y),(x',y')\in C}\left|d_X(x,x')-d_Y(y,y')\right|.
\end{equation}
The Gromov-Hausdorff distance between $(X,d_X)$ and $(Y,d_Y)$ is 
\begin{equation}
    \label{eq:GH}
    d_{GH}((X,d_X),(Y,d_Y)) = \frac{1}{2}\inf_{C\in\mathfrak{C}(X,Y)}\operatorname{dis}(C,d_X,d_Y).
\end{equation}
\end{definition}
When $(X,d_X)$ and $(Y,d_Y)$ are finite metric spaces, the sumpremum in  \cref{eq:distortion} is actually a maximum, and the infimum in  \cref{eq:GH} is a minimum. 
We refer the interested reader to \cite{burago01} for further details. 
\bigskip
\begin{remark}
From now on, we will denote by $\sif$ the set of strictly increasing functions $\psi:\mathbb{R}^+\longrightarrow\mathbb{R}^+$, with $\psi(0) = 0$.\\
\end{remark}

Consider now the map $\hat{d}: \mathrm{FMS}\times \mathrm{FMS}\longrightarrow \mathbb{R}^+$ given by
\begin{equation}
\label{eq:gh-diss}
\begin{split}
        \hat{d}((X,d_X),(Y,d_Y)) = &\inf_{\psi \in \sif}d_{GH}((X,\psi\circ d_{X}),(Y,d_Y)) +\\
        + &\inf_{\psi \in \sif}d_{GH}((X,d_{X}),(Y,\psi\circ d_{Y})).
\end{split}
\end{equation}
We have the following result
\begin{proposition}
The map $\hat{d}$ is a dissimilarity on the collection of $\mathrm{FMS}$. %
\end{proposition}
\begin{proof}
Since $d_{GH}$ is a distance, for all $(X,d_X)$ and $(Y,d_Y)$ it holds $$0\leq \inf_{\psi \in \sif}d_{GH}((X,\psi\circ d_{X}),(Y,d_Y)) < \infty$$ and 
$$0 \leq\inf_{\psi \in \sif}d_{GH}((X,d_{X}),(Y,\psi\circ d_{Y}))<\infty.$$ Therefore there exists $d_0=0$ such that  $d_0\leq\hat{d}(X,d_X),(Y,d_Y))<\infty$ for all $(X,d_X)$ and $(Y,d_Y)$. It is also possible to notice that $\hat{d}$ is symmetric by its very definition. Hence, $\hat{d}$ is a dissimilarity.
\end{proof}
Now, we show that that dissimilarity $\hat{d}$ satisfies our initial requests, especially the fact that is 0 if and only if two spaces are weakly isometric. 
\begin{proposition}
\label{prop:dtilde}
Given two \emph{FMS}s $(X,d_X)$, $(Y,d_Y)$, we have
\begin{equation*}
  \hat{d}((X,d_X),(Y,d_Y))=0 \iff (X,d_X)\wi(Y,d_Y).  
\end{equation*}
\end{proposition}
\begin{proof}
If $(X,d_X)\wi(Y,d_Y)$, there exists a strictly increasing function $\psi$ such that $(X,\psi\circ d_X)$ is isometric to $(Y,d_Y)$ and $(Y, \psi^{-1}\circ d_Y)$ is isometric to $(X,d_X)$. Since the Gromov-Hausdorff distance between two metric spaces is zero if and only if they are isometric, both the infima on the right hand side of \cref{eq:gh-diss} are $0$, hence $\hat{d}((X,d_X),(Y,d_Y))=0$.\\ 
On the other hand, suppose that $\hat{d}((X,d_X),(Y,d_Y))=0$, so that $$\inf_{\psi \in \sif}d_{GH}((X,\psi\circ d_{X}),(Y,d_Y))=0\,\,\mbox{and}\,\, \inf_{\psi \in\sif}d_{GH}((X,d_{X}),(Y,\psi\circ d_{Y}))=0.$$

Therefore, by the definition of infimum, there exist two sequences $(\psi_n)_{n\in\mathbb{N}} \subseteq \sif$ and $(\tilde{\psi}_n)_{n\in\mathbb{N}} \subseteq \sif$ such that
\begin{equation}
    \begin{split}
        \lim_{n\rightarrow \infty}d_{GH}((X,\psi_n\circ d_X),(Y,d_Y)) = 0 \\
        \lim_{n\rightarrow \infty}d_{GH}((X,d_X),(Y,\tilde{\psi}_n\circ d_Y)) = 0.
    \end{split}
\end{equation}
Let us focus our attention on the first sequence, $(\psi_n)_{n\in\mathbb{N}} \subseteq \sif$. By the finiteness of $\mathfrak{C}(X,Y)$, for every $n$ in $\mathbb{N}$ there exists a, possibly non-unique, correspondence $R_n$ in $\mathfrak{C}(X,Y)$ such that $\operatorname{dis}(R_n, \psi_n\circ d_X, d_Y)=d_{GH}((X,\psi_n\circ d_X),(Y,d_Y))$. By the axiom of choice, it is possible to construct a sequence $(\psi_n, R_n)_{n\in\mathbb{N}}\subseteq \sif\times\mathfrak{C}(X,Y)$ such that
$$
\lim_{n\to\infty}\operatorname{dis}(R_n, \psi_n\circ d_X, d_Y)=0.
$$
The set $\mathfrak{C}(X,Y)$ is finite, henceforth we can find a subsequence $(\hat{\psi}_n, \hat{R}_n)_{n\in\mathbb{N}}$ of $(\psi_n, R_n)_{n\in\mathbb{N}}$ such that there exists a $R_1$ in $\mathfrak{C}(X,Y)$ and a $\bar{n}$ in $\mathbb{N}$ with $\hat{R}_n=R_1$, for all $n\geq \bar{n}$. 

For this correspondence $R_1$, it holds
\begin{equation}
    \lim_{n\rightarrow\infty}|\hat{\psi}_n(d_X(x,x'))-d_Y(y,y')|=0 \quad \forall (x,y),(x',y')\in R_1.
\end{equation}
Hence, the restriction of the sequence $(\hat{\psi}_n)$ on the distance set $\mathrm{D}(X)$ converges to a function $\psi_X:\mathrm{D}(X)\longrightarrow \mathrm{D}(Y)$, such that $d_{GH}((X,\psi_X\circ d_X),(Y,d_Y))=0$. In the same way, we can prove the existence of a function $\psi_Y:\mathrm{D}(Y)\longrightarrow \mathrm{D}(X)$ such that $d_{GH}((X,d_X),(Y,\psi_Y\circ d_Y))=0$.

We observe that
\begin{equation}
\begin{split}
    d_{GH}((X,\psi_Y\circ\psi_X\circ d_X),(X,d_X))\leq d_{GH}((X,\psi_Y\circ\psi_X\circ d_X),(Y,\psi_Y\circ d_Y))+ \\ +d_{GH}((Y,\psi_Y\circ d_Y),(X,d_X)).  
\end{split}
\end{equation}
We already know that the second summand on the right hand side of the inequality is 0. For the first one it is easy to see that since $d_{GH}((X,\psi_X\circ d_X),(Y, d_Y))=0$, then also $d_{GH}((X,\psi_Y\circ\psi_X\circ d_X),(Y,\psi_Y\circ d_Y))=0$. Therefore, $d_{GH}((X,\psi_Y\circ\psi_X\circ d_X),(X,d_X))=0$, and $(X,\psi_Y\circ\psi_X\circ d_X)$ and $(X,d_X)$ are isometric. Since $\psi_X$ and $\psi_Y$ are both non decreasing functions, also their composition is non decreasing. Moreover, $\psi_Y\circ\psi_X$ has to be a bijective function from $\mathrm{D}(X)$ to itself, otherwise the two spaces fail to be isometric. Then, it has to be $\psi_Y\circ\psi_X = \operatorname{id}\vert_{\mathrm{D}(X)}$, therefore $\psi_X$ and $\psi_Y$ are invertible. It is possible to extend, by linear interpolation, the domain and codomain of $\psi_X$ to $\mathbb{R}^+$, thus we have a strictly increasing function $\psi_X$ such that $d_{GH}((X,\psi_X\circ d_X), (Y,d_Y))=0$ and this is equivalent to saying that $(X,d_X)$ and $(Y,d_Y)$ are weakly isometric.
\end{proof}

\begin{example}
In this example, we show the computation of $\tilde{d}$ between three finite metric spaces. Let us consider the spaces $(X,d_X)$, $(Y,d_Y)$ and $(Z,d_Z)$ depicted in  \cref{fig:example-dtilde}.
We can see that $\inf_{\psi\in\sif}d_{GH}((X,\psi\circ d_X), (Y,d_Y))=0$. In fact, if we take a sequence $(\psi_n)_{n\in\mathbb{N}}$ such that
\begin{equation*}
    \psi_n(3)=3,\ \psi_n(4)=4,\ \psi_n(5)=4+\frac{1}{n},
\end{equation*}
clearly $\lim_{n\to\infty} d_{GH}((X,\psi_n\circ d_X), (Y,d_Y))=0$. On the other hand, for $\hat{\psi}$ with
\begin{equation*}
    \hat{\psi}(3)=3,\ \hat{\psi}(4) = 4.5,
\end{equation*}
it holds $\inf_{\psi\in\sif}d_{GH}((X,d_X), (Y,\psi\circ d_Y))=d_{GH}((X,d_X), (Y,\hat{\psi}\circ d_Y))=0.5$. Therefore, $\tilde{d}((X,d_X),(Y,d_Y))=0.5$. Reasoning in a similar way it is possible to show that $\tilde{d}((Z,d_Z),(Y,d_Y))=1$. We also know by  \cref{prop:dtilde} that since $X$ and $Z$ are weakly isometric it has to be $\tilde{d}((X,d_X),(Z,d_Z))=0$. Hence,
\begin{equation*}
    \tilde{d}((Y,d_Y),(Z,d_Z))=1 > 0.5 = \tilde{d}((Y,d_Y),(X,d_X)) +\tilde{d}((X,d_X),(Z,d_Z))  
\end{equation*}
and the triangular inequality does not hold.
\end{example}

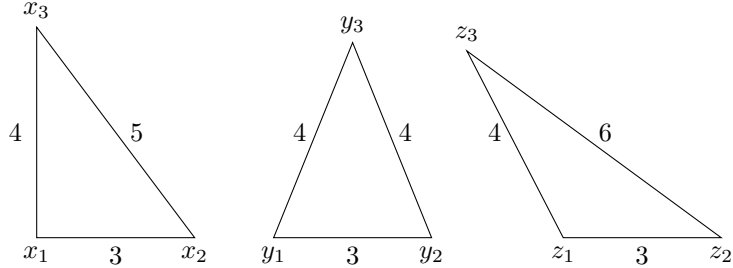
\begin{figure}[ht!]
    \centering
\begin{tikzpicture}[scale=0.7]
\draw (0,0) node[anchor=north]{$x_1$}
  -- (3,0) node[anchor=north]{$x_2$}
  -- (0,4) node[anchor=south]{$x_3$}
  -- cycle;
\draw (1.5, 0) node[anchor = north]{3};
\draw (1.6, 2) node[anchor =  west]{5};
\draw (-0.1, 2) node[anchor = east]{4};

\def\shiftx{4.5}
\draw (0+\shiftx,0) node[anchor=north]{$y_1$}
  -- (3+\shiftx,0) node[anchor=north]{$y_2$}
  -- (1.5+\shiftx,3.708) node[anchor=south]{$y_3$}
  -- cycle;
\draw (1.5+\shiftx, 0) node[anchor = north]{3};
\draw (0.8+\shiftx, 2) node[anchor =  east]{4};
\draw (2.2+\shiftx, 2) node[anchor = west]{4};

\def\shiftx{10}
\draw (0+\shiftx,0) node[anchor=north]{$z_1$}
  -- (3+\shiftx,0) node[anchor=north]{$z_2$}
  -- (-1.83+\shiftx,3.55) node[anchor=south]{$z_3$}
  -- cycle;
\draw (1.5+\shiftx, 0) node[anchor = north]{3};
\draw (-1+\shiftx, 2) node[anchor =  east]{4};
\draw (0.5+\shiftx, 2) node[anchor = west]{6};
\end{tikzpicture}
    \caption{Examples for the computation of $\hat{d}$.}
    \label{fig:example-dtilde}
\end{figure}

\section{Curvature sets of finite metric spaces}
\label{sec:curvature-sets}
Establishing whether two spaces are isometric or not is in general a computationally intensive problem. Hence, we would like to have a set of complete or incomplete invariants to study this problem. We will focus our attention firstly on the concept of curvature set introduced by Gromov in \cite{gromov07}. They have already been used by M\'emoli in \cite{memoli12}, they can in fact be used to obtain a lower bound for the Modified Gromov-Hausdorff distance between two metric spaces.  
\begin{definition}[curvature set - \cite{gromov07}]
\label{def:curvature-set}
Given a, non-necessarily finite, metric space $(X, d_X)$ we can consider the function
\begin{align}
\label{eq:def:fucntion-cset}
\begin{split}
\Psi_X^m :  X^{m} & \longrightarrow \mathbb{R}^{m \times m} \\
 (x_1, \dots , x_m) & \longmapsto M \text{ s.t. } M_{i,j} = d_X(x_i,x_j).
 \end{split}
\end{align}
We call $m$-th curvature set of $(X,d_X)$ the set 
\begin{equation}
\label{eq:def:curvature-set}
\mathrm{K}_m(X) := \textrm{im} \Psi_X^m.
\end{equation} 
\end{definition}
Let us see an example of some curvature sets for a finite metric space.
\begin{example}
\label{exmp:curvature-set}
Consider the set $X = \{x_1,x_2,x_3\}$ and endow it with the metric $d$ such that $d(x_1,x_2)=3$, $d(x_1,x_3) = 5$, $d(x_2,x_3) = 4$. Then, $\mathrm{K}_2(X)$ and $\mathrm{K}_3(x)$ are

\begin{equation*}
    \mathrm{K}_2(X) = \left\{
    \begin{bmatrix} 
        0 & 3\\
        3 & 0
    \end{bmatrix}, 
    \begin{bmatrix} 
        0 & 4\\
        4 & 0
    \end{bmatrix},
    \begin{bmatrix} 
        0 & 5\\
        5 & 0
    \end{bmatrix},
    \begin{bmatrix} 
        0 & 0\\
        0 & 0
    \end{bmatrix}
    \right\}.
\end{equation*}
\begin{equation*}
\begin{gathered}
    \mathrm{K}_3(X)=\Bigg\{
    P^T\begin{bmatrix}
    0 & 3 & 5 \\
    3 & 0 & 4 \\
    5 & 4 & 0
    \end{bmatrix}P,\  
    P^T\begin{bmatrix}
    0 & 0 & 3 \\
    0 & 0 & 3 \\
    3 & 3 & 0
    \end{bmatrix}P,\  
    P^T\begin{bmatrix}
    0 & 0 & 4 \\
    0 & 0 & 4 \\
    4 & 4 & 0
    \end{bmatrix}P,\ \\
    P^T\begin{bmatrix}
    0 & 0 & 5 \\
    0 & 0 & 5 \\
    5 & 5 & 0
    \end{bmatrix}P,\  
    \begin{bmatrix}
    0 & 0 & 0 \\
    0 & 0 & 0 \\
    0 & 0 & 0
    \end{bmatrix}\ \Bigg|\  
    P \textrm{ runs in $3\times 3$ permutation matrices}
    \Bigg\}.
\end{gathered}
\end{equation*}
\end{example}
Curvature sets encode information about finite metric subspaces of a given metric space. A similar, but coarser, invariant is the isometric sequence of a space, introduced by Hirasaka and Shinohara \cite{hirasaka18}. Curvature sets are important in virtue of the next theorem. 
\begin{theorem}[isometry of compact metric spaces - \cite{gromov07}] 
\label{thm:isometry-compact-ms}
Two compact metric spaces $X$ and $Y$ are isometric if and only if $\mathrm{K}_m(X) = \mathrm{K}_m(Y)$ for all $m \in \mathbb{N}$.
\end{theorem}

\begin{remark}
Notice that every FMS is compact and then satisfies the hypothesis of the above theorem. 
\end{remark}
Therefore, checking the equality of curvature sets is a way to find out whether two metric spaces are isometric or not. In the general case, in which a metric space is not finite, we have to prove the equality of all curvature sets to ensure the isometry between metric spaces, but for the finite case the problem becomes easier. We can see that the $r$-th curvature set carries all the information included in all the $l$-th curvature sets for $l < r$.
\begin{lemma}
\label{lem:curvature-set}
For any two, possibly infinite, metric spaces $(X,d_X)$, $(Y,d_Y)$, if $\mathrm{K}_r(X) = \mathrm{K}_r(Y)$ for a certain $r\in \mathbb{N}$ then $\mathrm{K}_l(X)=\mathrm{K}_l(Y)$ for all $l\leq r$.
\end{lemma}
\begin{proof}
Each matrix of $\mathrm{K}_l(X)$ can be obtained from a matrix of $\mathrm{K}_r(X)$ removing $r-l$ rows and columns. Given a set of indices $I := \{i_1, \dots, i_k\}\subseteq \{1, \dots, n\}$ and a matrix $M \in \mathbb{R}^{n\times n}$ we can define $M_I$ as the matrix obtained by $M$ removing the columns and the rows whose indices are in $I$. Then
\begin{equation}
    \begin{split}
    \mathrm{K}_l(X) =& \left\{M_I\ \middle|\ M\in \mathrm{K}_r(X),\ I\subseteq \{1,\dots,n\},\ |I|=r-l\right\} = \\
    =& \left\{M_I\ \middle|\ M\in \mathrm{K}_r(Y),\ I\subseteq \{1,\dots,n\},\ |I|=r-l\right\} = \mathrm{K}_l(Y).
    \end{split}
\end{equation}
\end{proof}
For finite metric spaces we can further improve this result. In the following theorem, we show that given a finite metric space $X$ of cardinality $n$, all the curvature sets are determined by $\mathrm{K}_n(X)$.

\begin{theorem}
\label{thm:m-th-curvature-set}
Let $(X, d_X)$ and $(Y, d_Y)$ be two finite metric space of cardinality $n$. Then
\begin{equation}
\mathrm{K}_m(X) = \mathrm{K}_m(Y) \quad \forall m \in \mathbb{N} \iff \mathrm{K}_n(X) = \mathrm{K}_n(Y).
\end{equation}
\end{theorem}

\begin{proof}
The forward implication is given by the hypothesis. We have to prove the other direction $\impliedby$. We already know, by  \cref{lem:curvature-set}, that $\mathrm{K}_m(X) = \mathrm{K}_m(Y)$ for all $m\leq n$. We have only to prove the case in which $m > n$. For every matrix $M \in \mathrm{K}_m(X)$, we can find $m$ points $x_1, \dots, x_m$ of $X$ such that $\Psi_X^m(x_1,\dots,x_m)=M$. Of these $m$ points at most $k \leq n$ of them can be different, let them be $x_{i_1}, \dots, x_{i_k}$. Then, the matrix $\Psi_X^k(x_{i_1}, \dots, x_{i_k})$ is in $\mathrm{K}_k(X)=\mathrm{K}_k(Y)$ by  \cref{lem:curvature-set}. Then, we have $y_1, \dots, y_k \in Y$ such that $\Psi_X^k(x_{i_1}, \dots, x_{i_k}) = \Psi_Y^k(y_1, \dots, y_k)$, and it means that
\begin{equation}
    \label{eq:thm-curvature-set-lower-dim}
    d_X(x_{i_a},x_{i_b}) = d_Y(y_a,y_b).
\end{equation}
Consider the bijection $\phi: \{x_{i_1}, \dots, x_{i_k} \}\longrightarrow\{y_1, \dots, y_k\}$ with $\phi(x_{i_j}) = y_j$, $j = 1,\dots, k$. Thanks to  \cref{eq:thm-curvature-set-lower-dim} we have that 
$$\Psi_X^m(x_1,\dots,x_m) = \Psi_Y^m(\phi(x_1),\dots,\phi(x_m)) \in \mathrm{K}_m(Y).$$ 
Therefore, $\mathrm{K}_m(X) \subseteq \mathrm{K}_m(Y)$. Analogously, we can see that $\mathrm{K}_m(Y) \subseteq \mathrm{K}_m(X)$, hence they must be equal. 
\end{proof}

We have seen in  \cref{exmp:curvature-set} that when we compute the $m$-th curvature set of a metric spaces we take $m$-tuples of points of $X$ with repetitions. This means computing $n^m$ matrices. We would like to reduce such a computational cost, and we try to do so introducing the concept of reduced curvature set.

\begin{definition}[reduced curvature set]
\label{def:reduced-curvature-set}
Consider a metric space $(X,d_X)$ and the associated function $\Psi_X^m :  X^{m} \longrightarrow \mathbb{R}^{m \times m}$ defined in \cref{eq:def:fucntion-cset}. We call $m$-th reduced curvature set of $(X,d_X)$ the set
\begin{equation}
\label{eq:def:reduced-curvature-set}
\tilde{\mathrm{K}}_m(X) = \left\{ \Psi_X^m(x_1, \dots, x_m)\ \middle|\ x_1,\dots, x_m \in X \textrm{ and } x_i \neq x_j \textrm{ if } i \neq j \right\}.
\end{equation}
\end{definition}
As we can see from the definition, to obtain the $m$-th reduced curvature set we need to compute $\frac{n!}{(n-m)!}$ matrices. We want to show that we do not lose any information with this reduction.
\begin{lemma}
\label{lem:reduced-curvature-set}
For any two metric spaces $(X,d_X)$, $(Y,d_Y)$, if $\tilde{\mathrm{K}}_r(X) = \tilde{\mathrm{K}}_r(Y)$ for a certain $r\in \mathbb{N}$, then $\tilde{\mathrm{K}}_l(X)=\tilde{\mathrm{K}}_l(Y)$ for all $l\leq r$.
\end{lemma}
\begin{proof}
The proof is analogous to that of   \cref{lem:curvature-set}.
\end{proof}

\begin{theorem}
\label{thm:equivalence-curvature-sets}
Let $(X, d_X)$ and $(Y, d_Y)$ be two finite metric spaces of cardinality $n$. For any $m \leq n$, we have
\begin{equation}
\mathrm{K}_m(X) = \mathrm{K}_m(Y) \iff \tilde{\mathrm{K}}_m(X) = \tilde{\mathrm{K}}_m(Y).
\end{equation}
\end{theorem}
\begin{proof}
Assume $\mathrm{K}_m(X) = \mathrm{K}_m(Y)$. Then a matrix $M \in \mathrm{K}_m(X)$ is also an element of $\tilde{\mathrm{K}}_m(X)$ if and only if $(M_{i,j} = 0 \iff i = j)$. The same argument holds also for $\mathrm{K}_m(Y)$, therefore, given $M \in \tilde{\mathrm{K}}_m(X)$, we have $M \in \mathrm{K}_m(X) = \mathrm{K}_m(Y)$ and $M\in \mathrm{K}_m(Y)$. Since $M$ has null entries only in its diagonal, $M \in \tilde{\mathrm{K}}_m(Y)$. In this way, we can see that $\tilde{\mathrm{K}}_m(X) \subseteq \tilde{\mathrm{K}}_m(Y)$ and $\tilde{\mathrm{K}}_m(Y) \subseteq \tilde{\mathrm{K}}_m(X)$, hence they are equal. Assume now that $\tilde{\mathrm{K}}_m(X) = \tilde{\mathrm{K}}_m(Y)$. We want to prove that, for any $M \in \mathrm{K}_m(X) \setminus \tilde{\mathrm{K}}_m(X)$, we have $M \in \mathrm{K}_m(Y)$. We know there are $x_1, \dots, x_m \in X$ such that $M = \Psi_X^m(x_1,\dots,x_m)$, where at most $k < m$ of the points are different. Suppose these points are $x_{i_1},\dots,x_{i_k}$. For  \cref{lem:reduced-curvature-set}, we have that $\tilde{\mathrm{K}}_{m-k}(X) = \tilde{\mathrm{K}}_{m-k}(Y)$, then we have $y_1, \dots. y_k$ points of $Y$ such that 
\begin{equation*}
    \Psi_X^{m-k}(x_{i_1},\dots, x_{i_k})= \Psi_Y^{m-k}(y_{1},\dots, y_{k}).
\end{equation*}
Hence, given the bijection $\phi: \{x_{i_1}, \dots, x_{i_k} \}\longrightarrow\{y_1, \dots, y_k\}$ with $\phi(x_{i_j}) = y_j$, $j = 1,\dots, k$, we have
\begin{equation}
    M = \Psi_X^{m}(x_{1},\dots, x_{m}) = \Psi_Y^{m}(\phi(x_{1}),\dots, \phi(x_{m})) \in \mathrm{K}_m(Y).
\end{equation}
Then, $\mathrm{K}_m(X)\subseteq \mathrm{K}_m(Y)$ and reasoning in the same way we have $\mathrm{K}_m(Y)\subseteq \mathrm{K}_m(X)$, therefore they are equal.
\end{proof}
Thanks to the last theorem, we can see that reduced curvature sets carry the same information given by the non reduced version. Notice that for finite metric spaces we do not have $m$-th reduced curvature sets, with $m$ greater than the number of points of the space. In the following theorem, we observe that isometry of finite metric spaces is characterised by the $n$-th reduced curvature set.

\begin{theorem}
Two finite metric spaces $(X, d_X)$ and $(Y, d_Y)$ of cardinality $n$ are isometric if and only if $\tilde{\mathrm{K}}_n(X) = \tilde{\mathrm{K}}_n(Y)$.
\end{theorem}

\begin{proof}
We need to prove the ``$\impliedby$'' implication only. Since $X$ and $Y$ are finite they are compact and by  \cref{thm:isometry-compact-ms} we know that they are isometric if and only if $\mathrm{K}_m(X) = \mathrm{K}_m(Y)$ for all $m \in \mathbb{N}$. By  \cref{thm:m-th-curvature-set}, since $X$ and $Y$ are finite, we know that this is true if and only if $\mathrm{K}_n(X) = \mathrm{K}_n(Y)$ and for  \cref{thm:equivalence-curvature-sets} this holds if and only if $\tilde{\mathrm{K}}_n(X) = \tilde{\mathrm{K}}_n(Y)$.
\end{proof}
We recall that, thanks to  \cref{thm:wi-to-isometry}, two spaces are weakly isometric if and only if their canonicalizations are isometric, and this condition can now be checked using the above theorem. Hence, we have the following corollary. 
\begin{corollary}
Two finite metric spaces $(X, d_X)$, $(Y, d_Y)$ of cardinality $n$ with respective canonicalizations $(\mathcal{C}_X, d_{\mathcal{C}_X})$, $(\mathcal{C}_{Y}, d_{\mathcal{C}_Y})$ are weakly isometric if and only if $\tilde{\mathrm{K}}_n(\mathcal{C}_X) = \tilde{\mathrm{K}}_n(\mathcal{C}_Y)$.
\end{corollary}

\section{Vietoris-Rips filtration and finite metric spaces}
\label{sec:vietoris-rips}
We will provide a categorification of the concept of weak isometry of finite metric spaces, introduced in the work \cite{ganyushkin94}, in order to obtain another complete invariant for weak isometry.
\begin{definition}[monotone map between FMS]
\label{def:monotone-map}
Given two finite metric spaces $(X,d_X)$ and $(Y,d_Y)$, we say the a map $f: X \longrightarrow Y$ is monotone if, for all $x_1, x_2, x'_1, x'_2 \in X$, we have:
\begin{equation}
    d_X(x_1,x_2) \leq d_X(x'_1,x'_2) \implies d_Y(f(x_1),f(x_2)) \leq d_Y(f(x'_1),f(x'_2)).
\end{equation}
\end{definition}

\begin{remark}[Proposition 3 - \cite{ganyushkin94}]
\label{rem:monotone-map}
If $f:X \longrightarrow Y$ is a monotone map, then
\begin{equation}
    d_X(x_1,x_2) = d_X(x'_1,x'_2) \implies d_Y(f(x_1),f(x_2)) = d_Y(f(x'_1),f(x'_2)).
\end{equation}
The converse is not true.
\end{remark}

\begin{lemma}
\label{lem:monotone-map}
A monotone map  $f:X \longrightarrow Y$ between two finite metric spaces $X$ and $Y$ induces a non-decreasing function between the distance sets $\tilde{f}: \spec(X)\longrightarrow \spec(Y)$ given by
\begin{equation}
    \tilde{f}(a) = d_Y(f(x_i),f(x_j))\quad \textrm{ where } d_X(x_i,x_j) = a. 
\end{equation}
\end{lemma}
\begin{proof}
Thanks to  \cref{rem:monotone-map}, we have that the function $\tilde{f}$ is well defined. In fact, for any $a \in \spec(X)$, we have that, if $d_X(x_i,x_j) = d_X(x'_i,x'_j) = a$, we can write  $\tilde{f}(a) = d_Y(f(x_i),f(x_j)) = d_Y(f(x'_i),f(x'_j))$. The function is non-decreasing because if $a = d_X(x_1,x_2) \leq b = d_X(x_3,x_4)$, by  \cref{def:monotone-map}, $\tilde{f}(a) = d_Y(f(x_1),f(x_2)) \leq d_Y(f(x_3),f(x_4)) = \tilde{f}(b)$.
\end{proof}

\begin{lemma}
\label{lem:extension-monotone-map}
A monotone map  $f:X \longrightarrow Y$ between two finite metric spaces $X$ and $Y$ induces a non-decreasing function $\hat{f}:\mathbb{R}^+\longrightarrow\mathbb{R}^+$ whose restriction is $\tilde{f}$ as in  \cref{lem:monotone-map}.
\end{lemma}
\begin{proof}
Thanks to  \cref{lem:monotone-map}, we have that $f$ induces a non-decreasing function $\tilde{f}:\spec(X)\longrightarrow\spec(Y)$. Such a function can be extended to a non decreasing function $\hat{f}:\mathbb{R}^+ \longrightarrow\mathbb{R}^+$ in the following way. If $\spec(X) = \left\{a_1, \dots, a_k\ \middle|\ a_i <a_j \textrm{ if } i <j\right\}$, we define $\hat{f}$ as
\begin{equation}
    \hat{f}(x) = 
    \begin{dcases}
    \frac{\tilde{f}(a_1)}{a_1}x \quad &\textrm{ if } x \in [0, a_1] \\
    \frac{\tilde{f}(a_{i+1})- \tilde{f}(a_{i})}{a_{i+1}-a_i}(x-a_i) + \tilde{f}(a_i) \quad &\textrm{ if } x \in [a_i, a_{i+1}] \\
    (x-a_k)+\tilde{f}(a_k) \quad &\textrm{ if } x \in (a_k,\infty).
    \end{dcases}
\end{equation}
By definition, it follows that $\hat{f}\big|_{\spec(X)} = \tilde{f}$. 
\end{proof}


\begin{definition}[category of FMS - \cite{ganyushkin94}]
\label{def:category-fms}
We can define a category $\fms$ of finite metric space whose objects are finite metric spaces and whose morphisms are monotone maps.
\end{definition}

\begin{remark}
In is possible to observe that two finite metric spaces are isomorphic in the category $\fms$ if and only if they are weakly isometric. 
\end{remark}

We recall some concepts of algebraic topology that we will use in the rest of this section. For more detailed information, we refer the reader to \cite{munkres84}.\\
An \emph{abstract simplicial complex} is a collection $K$ of finite non-empty sets, called \emph{simplices}, such that for any $\sigma $ in $K$ every $\tau$, non-empty subset of $\sigma$ is in $K$. If a simplex $\sigma$ has elements $x_0,\dots, x_k$, we will say that is generated by the points $x_0,\dots, x_k$ and we will write $\sigma  = \{x_0,\dots, x_k\}$.
The \emph{dimension} of a simplex is defined as the number of its elements minus 1, so
\begin{equation*}
    \textrm{dim}\sigma = \textrm{dim}\{x_0,\dots, x_k\} = k. 
\end{equation*}
The 0-dimensional simplices are also called \emph{vertices}. Given two abstract simplicial complexes $K$ and $L$, a \emph{simplicial map} $s: K\longrightarrow L$ is a function that sends the vertices of $K$ to vertices of $L$, such that, if $\sigma = \{x_0,\dots, x_k\}$ is a $k$-simplex of $K$, then $s(\sigma) = \{s(\{x_0\}), \dots, s(\{x_k\})\}$ is a simplex of $L$. Beware that the function $s$ does not need to be injective on the set of vertices. It may be possible that the image of a simplex $\sigma$ is a simplex $s(\sigma)$ of lower dimension.
\begin{example}
Consider the simplicial complex 
\begin{equation*}
    K =\left\{\{a\}, \{b\}, \{c\}, \{a,b\}, \{a,c\}, \{b,c\}, \{a,b,c\}\right\}.
\end{equation*}
An example of simplicial map is the function $s:K\longrightarrow K$, that on the set of vertices is equal to
\begin{equation}
    \begin{split}
        s(\{a\}) = \{a\}, \\
        s(\{b\}) = \{b\}, \\
        s(\{c\}) = \{b\}. \\
    \end{split}
\end{equation}
It is possible to see that some simplices degenerate, through $s$, to simplices of lower dimension. For example $s(\{a,b,c\}) = \{a,b\}$.
\end{example}
\begin{definition}[category of simplicial complexes]
We denote with $\mathbf{Simp}$ the category whose objects are finite abstract simplicial complexes, and morphisms are simplicial maps. 
\end{definition}
An example of abstract simplicial complex that we will use is the Vietoris-Rips complex.
\begin{definition}[Vietoris-Rips complex]
Given a finite metric space $(X,d_X)$ and a positive real number $\varepsilon$, the Vietoris-Rips complex $\vr_\varepsilon(X)$ is the abstract simplicial complex whose elements are subsets $\sigma = \{x_0, \dots, x_k\}$ of $X$ such that
\begin{equation*}
    \forall x_i,x_j \in \sigma, \quad d_X(x_i, x_j) \leq \varepsilon.
\end{equation*}
\end{definition}
We recall that we can associate to any partially ordered set $(P,\leq_P)$ the category whose objects are the elements of $P$, and whose morphisms are the relations $a \leq_P b$. In this section, we will consider in this way the category $(\mathbb{R}^+, \leq)$.
\begin{definition}[Vietoris-Rips filtration]
Given a finite metric space $(X,d_X)$, we can consider the functor $\vr_\bullet(X):(\mathbb{R}^+,\leq) \longrightarrow \mathbf{Simp}$ that assigns to each $a\in \mathbb{R}^+$ the Vietoris-Rips complex $\vr_a(X)$
and to each morphism $a\leq b$ the inclusion $\iota^X_{a\leq b}: \vr_a(X) \hookrightarrow \vr_b(X)$.
\end{definition}
\begin{definition}[rescaling]
Given a non-decreasing function $\psi:\mathbb{R}^+\longrightarrow\mathbb{R}^+$ we call $\psi$-rescaling the functor $R_\psi:(\mathbb{R}^+,\leq)\longrightarrow(\mathbb{R}^+,\leq)$ with
\begin{equation}
    \begin{split}
        R_\psi(a) &= \psi(a) \\
        R_\psi(a \leq b) &= \psi(a)\leq\psi(b).
    \end{split}
\end{equation}
\end{definition}

\begin{remark}[notation]
Given two functors $F:\mathbf{B}\longrightarrow\mathbf{C}$ and $G:\mathbf{C}\longrightarrow\mathbf{D}$ we denote their composition as $GF:\mathbf{B}\longrightarrow\mathbf{D}$. If we have another functor $F':\mathbf{B}\longrightarrow\mathbf{C}$, a natural transformation $\eta$ between $F$ and $F'$ is a family of morphisms $\{\eta_a :F(a)\longrightarrow F'(a)\ |\ a \in \textrm{ob}\left(\mathbf{B}\right) \}$ such that, for every morphism $m:a\longrightarrow b$ in $\mathbf{B}$, the following diagram commute
\begin{equation}
\begin{tikzcd}
F(a) \arrow[r, "F(m)"] \arrow[d, "\eta_a"] & F(b) \arrow[d, "\eta_b"] \\
F'(a) \arrow[r, "F'(m)"] & F'(b).
\end{tikzcd}
\end{equation}
In this case, we will write $\eta:F\nat F'$.
\end{remark}

\begin{lemma}
\label{lem:nat-transformation-monotone-map}
A morphism $f: X \longrightarrow Y$ in $\fms$ induces a rescaling $R_{\hat{f}}$ and a natural transformation $\eta^f:\vr_\bullet(X) \nat \vr_\bullet(Y)R_{\hat{f}}$.
\end{lemma}
\begin{proof}
Thanks to  \cref{lem:extension-monotone-map}, we have a non-decreasing function $\hat{f}$ that induces a rescaling $R_{\hat{f}}$. We want to see that, for any $a\in \mathbb{R}^+$, we have a simplicial map $\eta_a^f: \vr_{a}(X)\longrightarrow \vr_\bullet(Y)R_{\hat{f}}(a) = \vr_{\hat{f}(a)}(Y)$ such that, for all $a,b \in \mathbb{R}^+$ with $a \leq b$, we have a commutative diagram
\begin{equation}
\begin{tikzcd}
\vr_{a}(X) \arrow[r, hook, "\iota^X"] \arrow[d, "\eta_a^f"] & \vr_{b}(X) \arrow[d, "\eta_b^f"] \\
\vr_{\hat{f}(a)}(Y) \arrow[r, hook, "\iota^Y"] & \vr_{\hat{f}(b)}(Y).
\end{tikzcd}
\end{equation}
We define $\eta_a^f$ as
\begin{equation}
    \eta_a^f(\{x_{i_0},\dots, x_{i_k}\}) = \{f(x_{i_0}),\dots, f(x_{i_k})\}.
\end{equation}
By the very definition of Vietoris-Rips complex and $\hat{f}$, we have that, for every simplex $\sigma$ of $\vr_a(X)$, $\eta_a^f(\sigma)$ is a simplex of $\vr_{\hat{f}(a)}(Y)$ and $\eta_a^f$ is a well-defined simplicial map.
We only need to prove that $\iota^Y\circ \eta_a^f = \eta_b^f\circ \iota^X$. Indeed, for any $\sigma \in \vr_a(X)$ with $\sigma = \{x_{i_0},\dots, x_{i_k}\}$ we have
\begin{equation}
    \begin{split}
        \iota^Y\circ \eta_a^f(\sigma) = \iota^Y(\{f(x_{i_0}), \dots,f(x_{i_k})\})=&\ \{f(x_{i_0}), \dots,f(x_{i_k})\} \\
        \eta_b^f\circ\iota^X(\sigma) = \eta_b^f(\{x_{i_0},\dots ,x_{i_k}\}) =&\ \{f(x_{i_0}), \dots,f(x_{i_k})\}.
    \end{split}
\end{equation}
Hence $\eta^f$ is a natural transformation.
\end{proof}
Now, we want to show that the Vietoris-Rips filtration can be used as a complete invariant for weak isometry.
\begin{theorem}
\label{thm:vr-and-wi}
Given two finite metric spaces $X$ and $Y$, the following statements are equivalent:
\begin{enumerate}
    \item $X$ and $Y$ are isomorphic in $\fms$.
    \item There exist a rescaling $R_\psi$ and a natural isomorphism $$\eta: \vr_\bullet(X)\nat \vr_\bullet(Y)R_\psi.$$
\end{enumerate}
\end{theorem}
\begin{proof}
Suppose that $X$ and $Y$ are isomorphic in $\fms$. Then we have a monotone map $f:X\longrightarrow Y$ that is a bijection. Thanks to  \cref{lem:nat-transformation-monotone-map}, we have a rescaling $R_{\hat{f}}$ and a natural transformation $\eta^f:\vr_\bullet(X)\nat \vr_\bullet(Y)R_{\hat{f}}$. We want to see that for all $a \in \mathbb{R}^+$, $\eta_a^f$ is an isomorphism of simplicial complexes. Since $f$ is a bijection, we know that $\eta_a^f$ is injective. In fact, given $\sigma = \{x_{i_1}, \dots, x_{i_k}\}$ and $\tau = \{x_{j_1}, \dots, x_{j_l}\}$ with $\sigma \neq \tau$, we can assume without loss of generality that there is a $x_{j}\in \tau$ with $x_j \not\in \sigma$. Then if $\eta_a^f(\sigma) = \eta_a^f(\tau)$, there is a $x_i\in \sigma$ with $f(x_j)=f(x_i)$ and this is absurd for the injecivity of $f$. On the other hand $\eta_a^f$ is also surjective. Suppose that there is a $\rho \in \vr_{\hat{f}(a)}(Y)$ with $\rho = \{y_{i_1},\dots, y_{i_k}\}$ that is not in the image of $\eta_a^f$. We can consider the points $f^{-1}(y_{i_1}),\dots, f^{-1}(y_{i_k})$ of $X$ and see that they form a simplex of $\vr_a(X)$. In fact, since $\rho \in \vr_{\hat{f}(a)}(Y)$, for all $u,v \in \rho$, we have $d_Y(u,v)\leq \hat{f}(a)$. 
It can be seen that, for all $x_1, x_2 \in X$, we have $\hat{f}(d_X(x_1,x_2)) = d_Y(f(x_1),f(x_2))$. Therefore, for all $u,v \in \rho$, since $\hat{f}^{-1}$ is also a strictly increasing function
\begin{equation}
    d_X(f^{-1}(u), f^{-1}(v)) = \hat{f}^{-1}(d_Y(u,v))\leq \hat{f}^{-1}(\hat{f}(a)) = a.
\end{equation}
Then points $f^{-1}(y_{i_1}), \dots, f^{-1}(y_{i_k})$ span a simplex of $\vr_a(X)$ whose image under $\eta_a^f$ is $\rho$. Therefore $\eta_a^f$ is also bijective and is an isomorphism of simplicial complexes. Hence, $\eta$ is a natural isomorphism.\\
Now, suppose that we have a rescaling $R_\psi$ and a natural isomorphism $\eta:\vr_\bullet(X)\implies \vr_\bullet(Y)R_\psi$. For each $a \in \mathbb{R}^+$, the restriction of the isomorphism $\eta_a$ to the vertices of $\vr_a(X)$ yields a bijection $f_a:X\longrightarrow Y$. Moreover, these bijections are all the same because of the commutativity of diagrams that define the natural isomorphism. Hence we have a unique bijection $f:X\longrightarrow Y$ associated with $\eta$. We claim that this bijection is an isomorphism of finite metric spaces. In fact, for each couple of points $x_1,x_2 \in X$, call $\Bar{a} = d_X(x_1,x_2)$. Since $\eta$ is a natural isomorphism, we have that $\sigma = \{x_1,x_2\}\in \vr_{\Bar{a}}(X)$ and that $\eta_{\Bar{a}}(\sigma) = \{f(x_1),f(x_2)\}$ is a simplex of $\vr_{\psi(\Bar{a})}(Y)$, that is not present in any $\vr_b(Y)$, for $b\leq \psi(\Bar{a})$. This means that $d_Y(f(x_1),f(x_2)) = \psi(\Bar{a})= \psi(d_X(x_1,x_2))$, for all $x_1, x_2 \in X$, and therefore $X$ and $Y$ are weakly isometric and also isomorphic in $\fms$.
\end{proof}
\begin{remark}
\label{rem:ph-counterexample}
Given two Vietoris-Rips filtrations $\vr_\bullet(X)$ and $\vr_\bullet(Y)$, the existence of an isomorphism between each $\vr_a(X)$ and $\vr_a(Y)$ is not enough to ensure that $X$ and $Y$ are isometric. Indeed, all these isomorphism have to commute with the inclusions given by the filtrations. For example, consider the two metric spaces, depicted in  \cref{fig:ph-counterexample}, given by the distance matrices 
\begin{equation}
\begin{split}
    d_X&=(d_X(x_i,x_j))=\left(\begin{array}{cccc}
        0 & 7 & 9 & 10 \\
        7 & 0 & 8 & 11 \\
        9 & 8 & 0 & 12 \\
        10 & 11 & 12 & 0
    \end{array}\right),\\\quad\\
    d_Y&=(d_Y(y_i,y_j))=\left(\begin{array}{cccc}
        0 & 7 & 9 & 10 \\
        7 & 0 & 8 & 12 \\
        9 & 8 & 0 & 11 \\
        10 & 12 & 11 & 0
    \end{array}\right).
\end{split}
\end{equation}
\begin{figure}[ht!]
    \centering
\tdplotsetmaincoords{80}{150}
\begin{tikzpicture}[scale=0.4,tdplot_main_coords]
\def\shiftx{-1.3}
\def\shifty{15}
\draw[dashed,->] (0-\shiftx,0-\shifty,0) -- (2-\shiftx,0-\shifty,0) node[anchor=north east]{};
\draw[dashed,->] (0-\shiftx,0-\shifty,0) -- (0-\shiftx,2-\shifty,0) node[anchor=north west]{};
\draw[dashed,->] (0-\shiftx,0-\shifty,0) -- (0-\shiftx,0-\shifty,2) node[anchor=south]{};
\coordinate (O) at (0,0,0);
\coordinate (A) at (6.364,0,0);
\coordinate (B) at (7.989,6.809,0);
\coordinate (C) at (0,6.364,0);
\coordinate (D) at (2.855,1.205,9.286);
\coordinate (Dxy) at (2.855,1.205,0);
\draw[thick] (B) -- (A) node[anchor=north east]{$x_1=y_1$};
\draw[thick] (C) -- (B) node[anchor=north west]{$x_2=y_2$};
\draw[dashed] (A) -- (C) node[anchor=west]{$x_3=y_3$};
\draw[thick, blue] (A) -- (D) node[anchor=south]{$x_4$};
\draw[thick, blue] (B) -- (D);
\draw[thick, blue] (C) -- (D);

\def\shiftx{0}
\def\shifty{0}
\coordinate (A) at (6.364+\shiftx,0+\shifty,0);
\coordinate (B) at (7.989+\shiftx,6.809+\shifty,0);
\coordinate (C) at (0+\shiftx,6.364+\shifty,0);
\coordinate (D) at (5.677+\shiftx,2.220+\shifty,9.726);
\coordinate (Dxy) at (5.677+\shiftx,2.220+\shifty,0);
\draw[thick] (A) -- (D) node[anchor=south]{$y_4$};
\draw[thick] (B) -- (D);
\draw[thick] (C) -- (D);
\end{tikzpicture}
    \caption{Embedding of the spaces of  \cref{rem:ph-counterexample} in $\mathbb{R}^3$.}
    \label{fig:ph-counterexample}
\end{figure}
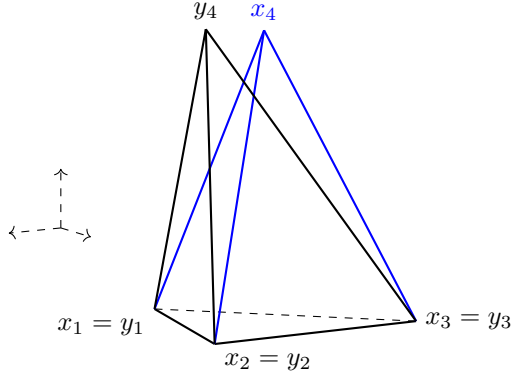

They are not isometric, but, for each $a \in \mathbb{R}^+$, we can find an isomorphism between $\vr_a(X)$ and $\vr_a(Y)$. Notice that, on the other hand, if there exists an $a$ in $\mathbb{R}^+$ such that there is no isomorphism between $\vr_a(X)$ and $\vr_a(Y)$, then the two spaces are for sure not isometric.
\end{remark}
\subsection{Persistent homology as an incomplete invariant for weak isometry}
Topological data analysis (TDA) is a branch of applied mathematics developed in the last 30 years in order to have a set of tools, based on topological and geometrical concepts, to analyse data \cite{frosini92, edelsbrunner08}. The main tool used in TDA is \emph{persistent homology}, an algebraic topology technique that consider the evolution and relations of homology groups of a sequence of nested topological spaces.
In the  usual pipeline of topological data analysis, the Vietoris-Rips filtration is used to compute the so called \emph{persistence module}, with which it is possible to keep track of the changes of homological features of the simplcial complexes at different scales of the filtration. 
Let us recall the definition of the $k$-th \emph{degree simplicial homology functor}. Given a field $\mathbb{F}$, with $\textbf{Vect}_{\mathbb{F}}$ we denote the category whose objects are finite dimensional vectors spaces over $\mathbb{F}$ and with morphisms the linear maps between these spaces.
Let us consider an abstract simplicial complex $K$. To simplify the definitions we will consider a total order on the set of vertices of $K$. Every $k$-simplex of $K$ from now on will be considered as an ordered tuple $[x_0,\dots,x_k]$ of $k+1$ vertices of $K$.
The group of $k$-chains of $K$, denoted with $C_k(K)$ is the vectors space whose elements are formal sums
\begin{equation}
\sum_i \lambda_i\sigma_i
\end{equation}
where $\lambda_i$ is an element of $\mathbb{F}$ and $\sigma_i$ is a $k$-simplex of $K$. It is possible to see that a basis is given by the formal sums $\{1_{\mathbb{F}}\sigma\ |\ \text{dim}\sigma=k\}$. 
Between every $C_k(K)$ and $C_{k-1}(K)$, there is a linear map $\partial_k:C_{k}(K)\to C_{k-1}(K)$ called \emph{boundary operator}. It is defined on the basis of $C_k$ given by the set of $k$-simplices as
\begin{equation}
    \partial_k([x_0,\dots,x_k])= \sum_{i=0}^k (-1)^i[x_0,\dots,\hat{x_i},\dots, x_k],
\end{equation}
where with $[x_0,\dots,\hat{x_i},\dots, x_k]$ we denote the $(k-1)$-simplex obtained removing the vertex $x_i$. It is possible to show that for every $k>0$ it holds $\partial_k\circ\partial_{k+1}=0$, meaning that the set $\operatorname{im}\partial_{k+1}$ is a vector subspace of $\operatorname{ker}\partial_k$. Then, the $k$-th homology group of $K$ with coefficients in $\mathbb{F}$ is the quotient vector space 
\begin{equation}
    H_k(K):= \frac{\operatorname{ker}\partial_k}{\operatorname{im}\partial_{k+1}}.
\end{equation}
The $k$-th simplicial homology functor is a functor from the category \textbf{Simp} to the category $\textbf{Vect}_{\mathbb{F}}$, that assigns to every object $K$ of \textbf{Simp} the homology group $H_k(K)$ in $\textbf{Vect}_{\mathbb{F}}$, and to every simplicial map $s:K\to L$ the induced map $H_k(s):H_K(K)\to H_k(L)$ between the homology groups.\\

In general therms, a \emph{filtration} is a functor $F: (\mathbb{R_+},\leq)\to \textbf{Simp}$. Given a filtration $F$, the composition $H_kF$ is called the $k$-th degree \emph{persistence module}. It is possible to define an \emph{interleaving distance} between persistence modules in the following way, as described in detail in \cite{bubenik14}.
Given a positive real number $\varepsilon$, it is possible to define a translation functor $T_\varepsilon:(\mathbb{R}^+, \leq)\longrightarrow(\mathbb{R}^+, \leq)$, with $T_\varepsilon(x)=x+\varepsilon$ and a natural transformation $\nu_\varepsilon: \operatorname{Id}_{(\mathbb{R}^+,\leq)}\nat T_\varepsilon$, with $\nu_\varepsilon(x):x \longrightarrow x+\varepsilon$ defined as $x\leq x+\varepsilon$.
Two persistence modules $H_kF$ and $H_kG$ are said to be $\varepsilon$-iterleaved if there exist two natural transformation $\eta_F:H_kF\nat H_kGT_\varepsilon$ and $\eta_G:H_kG\nat H_kFT_\varepsilon$ such that
\begin{equation}
    (\eta_GT_\varepsilon)\eta_F = H_kF\nu_{2\varepsilon}\quad \text{ and }\quad(\eta_FT_\varepsilon)\eta_G = H_kG\nu_{2\varepsilon}.
\end{equation}
The interleaving distance $d_I$ between two persistence modules $H_kF$ and $H_kG$ is defined as
\begin{equation}
    \label{eq:interleaving-distance}
    d_I(H_kF, H_kG)=\inf\left\{\varepsilon\geq 0\ |\ H_kF\text{ and }H_kG \text{ are $\varepsilon$-interleaved}\right\}.
\end{equation}
The reader familiar with  persistent homology will find  the following corollary to be natural.
\begin{corollary}
The persistent homology of the Vietoris-Rips filtration is an incomplete invariant for isometry.
\end{corollary}
\begin{proof}
The two spaces in  \cref{rem:ph-counterexample} have persistence modules with interleaving distance 0, yet they are not isometric.
\end{proof}
Even if persistent homology is only an incomplete invariant, we can still use it to study the problem of weak isometry. Given two finite metric spaces, it is possible to compute their canonicalizations, from them their associated Vietoris-Rips filtrations and then the persistence modules. If the obtained persistence modules have interleaving distance greater than 0, then the two spaces cannot be weakly isometric.

\subsection{A dissimilarity measure for persistence modules}
In a spirit similar to  \cref{sec:dissimilarity}, we can define a dissimilarity between persistence modules. Recall that we defined $\sif$ as the set of strictly increasing functions $\psi:\mathbb{R}^+\longrightarrow\mathbb{R}^+$, with $\psi(0) = 0$.
\begin{definition}
Given two persistence modules $HF_1$ and $HF_2$, the dissimilarity $\tilde{d}$ between them is
\begin{equation}
    \tilde{d}(HF_1,HF_2) = \inf_{\psi\in\sif}d_I(HF_1,HF_2R_\psi) + \inf_{\psi\in\sif}d_I(HF_1R_\psi,HF_2), 
\end{equation}
where $d_I$ is the interleaving distance between persistence modules.
\end{definition}
One of the key properties desired for distances between persistence modules is that they satisfy a form of stability theorem, that is, there must be a distance between the original metric spaces, that bounds from above the distance between the obtained persistence modules. We have a stability theorem of this kind for the Vietoris-Rips filtration and the interleaving distance \cite{chazal14}, in fact
\begin{equation*}
    d_{I}(H_k\vr_\bullet(X,d_X),H_k\vr_\bullet(Y,d_Y))\leq 2d_{GH}((X,d_X),(Y,d_Y)).
\end{equation*}
We will provide a similar stability theorem also for the distances introduced in this paper. 
\begin{theorem}[stability theorem for weak isometry]
\label{thm:weak-stability-theorem}
Let $(X,d_X)$, $(Y,d_Y)$ be two finite metric spaces. We denote with $H_k\vr_\bullet(X,d_X)$ and $H_k\vr_\bullet(Y,d_Y)$ the $k$-th persistence modules obtained from the Vietoris-Rips filtration associated with the two spaces. Then for all $k\in \mathbb{N}$,
\begin{equation}
    \label{eq:wi-stability}
    \tilde{d}(H_k\vr_\bullet(X,d_X),H_k\vr_\bullet(Y,d_Y))\leq 2\hat{d}((X,d_X),(Y,d_Y)).
\end{equation}
\end{theorem}

We can see that thanks to this theorem and  \cref{prop:dtilde}, if the persistence modules obtained by two Vietoris-Rips filtration have distance $\tilde{d}$ greater than 0, then the corresponding finite metric spaces cannot be weakly isometric.

\begin{example}
\label{exmp:dtilde}
Consider the finite metric spaces $(X,d_X)$ and $(Y,d_Y)$ with
\begin{equation}
    \begin{split}
    d_X&=(d_X(x_i,x_j))=\left(\begin{array}{cccc}
        0 & 7 & 12 & 8 \\
        7 & 0 & 10 & 11 \\
        12 & 10 & 0 & 9 \\
        8 & 11 & 9 & 0
    \end{array}\right),\\\quad\\
    d_Y&=(d_Y(y_i,y_j))=\left(\begin{array}{cccc}
        0 & 7 & 12 & 8 \\
        7 & 0 & 10 & 9 \\
        12 & 10 & 0 & 11 \\
        8 & 9 & 11 & 0
    \end{array}\right).
\end{split}
\end{equation}

\begin{figure}[ht!]
    \centering
\tdplotsetmaincoords{80}{290}
\begin{tikzpicture}[scale=0.4,tdplot_main_coords]
\def\shiftx{-1.3}
\def\shifty{10}
\draw[dashed,->] (0-\shiftx,0-\shifty,0) -- (2-\shiftx,0-\shifty,0) node[anchor=north west]{};
\draw[dashed,->] (0-\shiftx,0-\shifty,0) -- (0-\shiftx,2-\shifty,0) node[anchor=north east]{};
\draw[dashed,->] (0-\shiftx,0-\shifty,0) -- (0-\shiftx,0-\shifty,2) node[anchor=south]{};
\coordinate (A) at (0,0,0);
\coordinate (B) at (7,0,0);
\coordinate (C) at (6.64,9.99,0);
\coordinate (D) at (-0.57,6.73,4.29);
\coordinate (D2) at (2.29,2.83,7.13);
\draw[thick] (B) -- (A) node[anchor=north east]{$x_1=y_1$};
\draw[dashed] (C) -- (B) node[anchor=north west]{$x_2=y_2$};
\draw[thick] (A) -- (C) node[anchor=east]{$x_3=y_3$};
\draw[thick, blue] (A) -- (D) node[anchor=south]{$x_4$};
\draw[thick, blue] (B) -- (D);
\draw[thick, blue] (C) -- (D);
\draw[thick] (A) -- (D2) node[anchor=south]{$y_4$};
\draw[thick] (B) -- (D2);
\draw[thick] (C) -- (D2);
\end{tikzpicture}
    \caption{Embedding of the spaces of  \cref{exmp:dtilde} in $\mathbb{R}^3$.}
\end{figure}
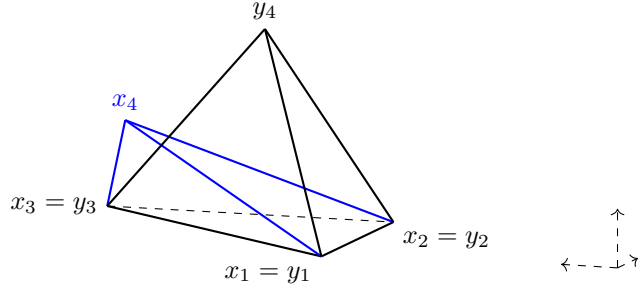

We will use $\mathbb{Z}_2$ as the field of coefficients with which we will compute homology.
We can see that the persistence module $H_1\vr_\bullet(X,d_X)$ is the functor
\begin{equation*}
    \begin{gathered}
    H_1\vr_\varepsilon(X,d_X)= 
    \begin{cases}
    \mathbb{Z}_2\quad &\text{ if } 10\leq\varepsilon\leq 11\\
    0\quad &\text{ otherwise},
    \end{cases}\\
    H_1\vr_\bullet(X,d_X)(a\leq b) = \\ =
    \begin{cases}
    \operatorname{id}:\mathbb{Z}_2\to \mathbb{Z}_2 &\text{ if } 10\leq a \leq b \leq 11\\
    0: H_1\vr_a(X,d_X)\to H_1\vr_b(X,d_X) &\text{ otherwise}.
    \end{cases}
    \end{gathered}
\end{equation*}
The persistence module $H_1\vr_\bullet(Y,d_Y)$ is
\begin{equation*}
    \begin{gathered}
    H_1\vr_\varepsilon(Y,d_Y)= 0\quad \forall \varepsilon \geq 0\\
    H_1\vr_\bullet(Y,d_Y)(a\leq b) = 0: H_1\vr_a(Y,d_Y)\to H_1\vr_b(Y,d_Y)\quad \forall a\leq b.
    \end{gathered}
\end{equation*}
Let us consider the sequence $(\psi_n)_{n\in \mathbb{N}}$ with
\begin{equation*}
    \psi_n(x) = \begin{cases}
    x \quad &\text{ if }0\leq x \leq 10 \\
    n(x-10)+10 &\text{ if } 10< x \leq 10+\frac{1}{n}\\
    x+1-\frac{1}{n} \quad &\text{ if }10+\frac{1}{n}< x.
    \end{cases}
\end{equation*}
It is easy to see that 
\begin{equation*}
    \lim_{n\to\infty}d_I(H_1\vr_\bullet(Y,d_Y),H_1\vr_\bullet(X,d_X)R_{\psi_n}) = 0.
\end{equation*}
On the other hand, for all  strictly increasing functions $\psi:\mathbb{R}^+\to\mathbb{R}^+$ it holds
\begin{equation*}
d_I(H_1\vr_\bullet(X,d_X),H_1\vr_\bullet(Y,d_Y)R_\psi) = \frac{1}{2}. 
\end{equation*}
Therefore, $\tilde{d}(H_1\vr_\bullet(X,d_X),H_1\vr_\bullet(Y,d_Y)) = \frac{1}{2}>0$, and the two spaces are not weakly isometric.
\end{example}

\begin{proof}[Proof of  \cref{thm:weak-stability-theorem}]
We have that, for every $\psi \in \sif$,
\begin{equation*}
    H_k\vr_\bullet(X,\psi\circ d_X))R_\psi= H_k\vr_\bullet(X,d_X).
\end{equation*}
This is true because $(X,\psi\circ d_X)$ and $(X,d_X)$ are weakly isometric and, as in the proof of  \cref{thm:vr-and-wi}, there are a natural isomorphism $\eta$ and a rescaling $R_\psi$ such that $\eta:\vr_\bullet(X, d_X)\nat \vr_\bullet(X,\psi\circ d_X)R_\psi$. In this case, the natural isomorphism is simply the identity between every $\vr_a(X,d_X)$ and $\vr_\psi(a)(X,\psi\circ d_X)$. Therefore, the two functors $\vr_\bullet(X,d_X)$ and $\vr_\bullet(X,\psi\circ d_X)R_\psi$ are equal.
\newline
Now, let us take a sequence $(\psi_n)_{n\in\mathbb{N}}$ in $\sif$ with 
$$\lim_{n\rightarrow\infty}d_{GH}((X,\psi_n\circ d_X), (Y,d_Y)) = \inf_{\psi\in\sif}d_{GH}((X,\psi\circ d_X),(Y, d_Y))$$
and a sequence $(\bar{\psi}_n)_{n\in\mathbb{N}}$ with $$\lim_{n\rightarrow\infty}d_{GH}((X,d_X),(Y,\bar{\psi}_n\circ d_Y)) = \inf_{\psi\in\sif}d_{GH}((X,d_X),(Y,\psi\circ d_Y)).$$
By the classical stability theorem \cite{chazal14}, for all $n$ in $\mathbb{N}$, we have
\begin{equation}
\label{eq:inequality-stab-thm}
    \begin{split}
        d_I(H_k\vr\bullet(X,\psi_n\circ d_X),H_k\vr\bullet(Y,d_Y))\leq 2d_{GH}((X,\psi_n\circ d_X),(Y, d_Y)),\\
        d_I(H_k\vr\bullet(X,d_X),H_k\vr\bullet(Y,\bar{\psi}_n\circ d_Y))\leq 2d_{GH}((X,d_X), (Y,\bar{\psi}_n\circ d_Y)).
    \end{split}
\end{equation}
Recall that, for all $n$ in $\mathbb{N}$ by the definition of infimum
\begin{equation}
\label{eq:infima-interleaving-d}
    \begin{gathered}
        \inf_{\psi\in \sif}d_I(H_k\vr\bullet(X,\psi\circ d_X),H_k\vr\bullet(Y,d_Y))\leq\\\leq d_I(H_k\vr\bullet(X,\psi_n\circ d_X),H_k\vr\bullet(Y,d_Y)),\\
        \text{and}\\
        \inf_{\psi\in\sif}d_I(H_k\vr\bullet(X,d_X),H_k\vr\bullet(Y,\psi\circ d_Y))\leq \\
        \leq d_I(H_k\vr\bullet(X,d_X),H_k\vr\bullet(Y,\bar{\psi}_n\circ d_Y)).
    \end{gathered}
\end{equation}
By the definition of $\tilde{d}$ and because of the previous inequalities it holds
\begin{equation}
\begin{gathered}
    \tilde{d}(H_k\vr\bullet(X,d_X),H_k\vr\bullet(Y,d_Y))\leq\\
    \lim_{n\rightarrow\infty}(d_I(H_k\vr\bullet(X,\psi_n\circ d_X),H_k\vr\bullet(Y,d_Y))+\\
    +d_I(H_k\vr\bullet(X,d_X),H_k\vr\bullet(Y,\bar{\psi}_n\circ d_Y)).
\end{gathered}
\end{equation}
Because of the inequalities in  \cref{eq:inequality-stab-thm} and the definition of $\hat{d}$, the following is true
\begin{equation}
\begin{gathered}
    \lim_{n\rightarrow\infty}(d_I(H_k\vr\bullet(X,\psi_n\circ d_X),H_k\vr\bullet(Y,d_Y))+\\
    +d_I(H_k\vr\bullet(X,d_X),H_k\vr\bullet(Y,\bar{\psi}_n\circ d_Y))\leq\\
    \leq\lim_{n\rightarrow\infty}(2d_{GH}((X,\psi_n\circ d_X), (Y,d_Y))+2d_{GH}((X,d_X), (Y,\bar{\psi}_n\circ d_Y)) =\\= 2\hat{d}((X,d_X),(Y,d_Y)).
\end{gathered}
\end{equation}
Therefore,
\begin{equation}
    \tilde{d}(H_k\vr\bullet(X,d_X),H_k\vr\bullet(Y,d_Y))\leq 2\hat{d}((X,d_X),(Y,d_Y)).
\end{equation}
\end{proof}

\section{Conclusions and future work}
\label{sec:conclusions}
We have constructed suitable representative elements for the equivalence classes of the relation of weak isometry for finite metric spaces. Thanks to these representatives we can check in a simple way whether two spaces are weakly isometric or not. We have given a definition of a notion of dissimilarity between finite metric spaces which measures how far they are from being weakly isometric. We have shown that curvature sets and Vietoris-Rips filtrations can help us in characterizing the classes of weak isometry. We have seen that we can use persistent homology to try to discriminate non-weakly isometric finite metric spaces, and we have defined a dissimilarity between persistence modules and shown that it satisfies a stability theorem for weak isometry. We would like to point out that in this work we never exploited the property of triangle inequality, therefore all the results could be easily extended to finite semi-metric spaces (see \cite{dovgoshey13}).   
In the future, we would like to try to find other and simpler invariants for weak isometry and make a comparison between them. We have also seen the usefulness of persistent homology, and in future work we would like to study the problem of finding distances between persistence modules that are meaningful from the point of view of weak isometry.

\section*{Acknowledgement}
AD, UF and FV have been supported by the SmartData@PoliTO center on Big Data and Data Science and by the Italian MIUR Award “Dipartimento di Eccellenza 2018-2022” - CUP: E11G18000350001.

\bibliographystyle{plainurl}
\bibliography{bibliography}

\end{document}